\documentclass[11pt]{article}
\usepackage{graphicx}

\usepackage[utf8]{inputenc}
\usepackage{amsmath,amsfonts,amssymb,amscd,amsthm,txfonts,hyperref}

\newtheorem{theorem}{Theorem}[section]
\newtheorem{proposition}[theorem]{Proposition}
\newtheorem{lemma}[theorem]{Lemma}

\newtheorem{remark}{Remark}[section]
\newtheorem{theo}{Theorem}

\theoremstyle{definition}
\newtheorem{definition}[theorem]{Definition}

\newcommand{\R}{\varmathbb R}

\newcommand{\N}{\varmathbb N}

\newcommand{\Hs}{\mathcal H}

\newcommand{\E}{\mathcal E}
\newcommand{\reff}[1]{(\ref{#1})}

\renewcommand{\Re}{\mbox{Re}}

\title{Invariant measure for the cubic wave equation on the unit ball of $\R^3$}
\author{Anne-Sophie de Suzzoni\footnote{D\'epartement de Math\'ematiques, Universit\'e de Cergy-Pontoise, Site de Saint Martin, 2, av Adolphe Chauvin, 95302 Cergy-Pontoise Cedex, FRANCE, e-mail : \texttt{anne-sophie.de-suzzoni@u-cergy.fr}}}

\begin{document}

\maketitle

\begin{abstract}This paper deals with the invariance of a measure on Sobolev spaces of low regularity under the flow of the cubic non linear wave equation on the unit ball of $\R^3$ under the assumption of spherical symmetry. It presents two aspects, an analytic one which includes the treatment of local properties of the flow, and a probabilistic one, which is mainly related to the global extension of the flow and the invariance of the measure.\end{abstract}

\tableofcontents

\section{Introduction}

The goal here is to prove the invariance of the measure $\rho$ constructed in \cite{wavbt}, II, under the flow of the cubic non linear wave equation on the unit ball of $\R^3$, hence answering the remark 6.3 of the paper by Nicolas Burq and Nikolay Tzvetkov.

\smallskip

The equation studied is : 

\begin{equation}\label{real} \left \lbrace{ \begin{tabular}{ll}
$\partial_t^2 f - \Delta_{B^3}f + f^3 = 0$ & $(t,x) \in \R \times B^3$ \\
$f|_{t=0} = f_0$ & $\partial_t f|_{t=0} = f_1 $ \end{tabular} } \right. \; , \end{equation}
where $f$ is real, radial, $B^3$ is the unit ball in $\R^3$, and $\Delta_{B^3}$ is the Laplace-Beltrami operator on $B^3$ with Dirichlet boundary conditions. Though, it is soon to be changed into its complex form, that is, writing $H = \sqrt{ -\Delta_{B^3}}$ and $u = f-iH^{-1}\partial_t f$,

\begin{equation}\label{compl} \left \lbrace{ \begin{tabular}{ll}
$i\partial_t u + Hu + H^{-1}(\Re u)^3 = 0$ \\
$u|_{t=0} = u_0 = f_0 -iH^{-1}f_1 $ \end{tabular} } \right. \; .\end{equation}

In order to define the measure invariant under the flow of \reff{compl}, a sequence of independant complex centered and normalised (law $\mathcal N (0,1)$)  $(g_n)_n$ is introduced, along with the measure $\mu$, which is the image measure of the (well-defined) map from a probability space to the Sobolev space $H^\sigma$, $\sigma<\frac{1}{2}$ : 

$$\varphi(\omega, r) = \sum_{n=1}^\infty \frac{g_n}{\pi n} e_n$$
where $e_n$ are the radial eigenfunctions of $\Delta_{B^3}$ with eigenvalues $\pi^2n^2$. It makes $\mu$ a sort of limit of gaussian on $\R^N$ when $N$ goes to infinity.

\smallskip

The measure $\rho$ is then defined as : 

$$d\rho (u) = e^{-\frac{1}{4}||u||_{L^4(B^3)}^4}d\mu(u)$$
absolutely continuous wrt $\mu$. It has been proved that $\rho$ is genuine, the norm $||\; .\; ||_{L^4(B^3)}$ being $\mu$ - almost surely finite. 

\smallskip

It comes from \cite{wavbt} :

\begin{theo}[Burq,Tzvetkov] Let $\sigma< \frac{1}{2}$. There exists a set $\Sigma \subseteq H^\sigma$ of full $\mu$ or $\rho$ (which is equivalent) measure such that for any initial data $u_0$, the flow is globally well-defined and what is more the solution of \reff{compl} is unique in $S(t)u_0 + H^s$ where $S(t)$ is the flow of the linear equation $i\partial_t u + H u = 0$ and $s$ is some real number $s>\frac{1}{2}$. \end{theo}

Using the ideas of the proof of this theorem and the local property of the solution, the following theorem will be proved.

\begin{theo}\label{tttoo} There exists a set $\Pi \subseteq H^\sigma$ of full $\rho$ measure such that the solution of \reff{compl} is strongly globally well-defined for any initial data taken in $\Pi$ and that all measurable set $A$ included in $\Pi$ satisfies at all time $t$ : 

$$\rho(\psi(t) A) = \rho(A) \; .$$\end{theo}

Before going further, it has to be understood that $\rho$ is built to be invariant under the flow of the non linear wave equation. Actually, by applying a cut-off on the frequency of the Lapace-Beltrami operator with radial symmetry and Dirichlet boundary conditions of the unit ball in $\R^3$, the NLW is approached by PDE in finite dimension, suceptible to finite dimension theory, like Cauchy-Lipschitz theorem. Indeed, call $E_N$ the space linearly spanned by $N$ first eigen functions of the radial Laplace-Beltrami operator on the unit ball of $\R^3$ with Dirichlet boundary conditions, by using projectors on $E_N$, or, better to say, operators that send $H^s$ into $E_N$ which are more regular than mere orthogonal projectors, the non linear wave equation can be reduced onto a problem on $E_N$, which admits a unique maximal condition thanks to Cauchy-Lipschitz theorem that can be proved to be a global one thanks to the existence of a conserved positive energy. The reduction is chosen such that the solutions converges in the space of distributions towards a solution of the non linear wave equation.

\smallskip

Then, finding a measure $\rho_N$ on $E_N$ which is invariant under the flow of this equation relies mostly on the existence of a conserved energy and Liouville theorem. It happens that the sequence $\rho_N$ extended to $H^\sigma$ converges towards a non trivial measure $\rho$ on $H^\sigma $. The measure $\rho$ being a limit of $\rho_N$, it is expected to be invariant under the flow of NLW. 

\smallskip

In previous works, like \cite{artbou,bourgainfin,bourtwod,linear}, the strategy applied to prove the invariance of the measure in infinite dimension used the fact that the initial datum was taken in spaces $\Hs$ such that the orthogonal projectors $\Pi_N$ on $E_N$ were uniformly bounded, that is to say, there exists $C$ independant from $N$, such that $||\Pi_N||_{\Hs\rightarrow \Hs} \leq C$. So the projections $\Pi_N u$ converged toward $u$ in $\Hs$ uniformly in every compact subset of $\Hs$. In these cases, by approaching the the initial data in $\Hs$ by its projections, the flow with initial data $u_0$ could be approached by the finite dimensional flow of $\Pi_N u_0$, the convergence being uniform on any compact set of initial data in $\Hs$. However, here, the control that ensures the existence of global strong solution in \cite{wavbt} is the norm :

$$||S(t)u_0||_{L^p_{t\in [0,2],x\in B^3}} \; ,$$
and thus the convergence of $\Pi_N u_0$ in $H^\sigma$ does not ensure the convergence of this norm applied to $\Pi_N u_0$ and thus not the uniform (regarding the initial data) convergence of finite dimensional solutions towards the global solution.

\smallskip

This problem can be solved though by introducing slightly different ``finite" measures and dimensional equations. Instead of entirely reducing the problem to a problem on $E_N$, only its non linear part will, that is to say, the initial data will be taken in $H^\sigma$ but the non linear part will be projected on $E_N$. Therefore, the reduced problem will present two parts : a linear and of infinite dimension one and a finite dimensional though non linear one. Then, the reduced measure $\rho_N$, instead of being defined on $E_N$ will be defined on all $H^\sigma$ (and still invariant under the flow). The initial data, thus, will not have to be approached, only the flow will, like in \cite{lestrois}. Unlike in \cite{lestrois} though, considerations on the reversibility of the flow will be confined to the linear treatment. Indeed, the flow of the linear equation is defined on all $H^\sigma$ which makes it easier to manipulate. The afore-mentioned strategy using the uniform bound of the projectors will be used to prove the invariance of $\mu$ under the linear flow. Nevertheless, the flow of the NLW being defined only on a subset $\Sigma$ of $H^s$, this subset has to be invariant under the flow if the reversibility of the flow must be used.

\smallskip

To sum up, \cite{wavbt} will provide the topological framework and the local results of existence for the non linear wave equation, \cite{lestrois} the descriptions of the ``new" partly finite dimensional measures, \cite{linear} a guideline to prove the invariance of the measure $\mu$ under the linear flow, \cite{gaussmea} the main ideas and properties about random gaussian series, and thanks to all these results, the invariance of  $\rho$ shall be proved.

\smallskip

\paragraph{Plan of the paper. }The first part is a reminder of the results of \cite{wavbt,linear,lestrois} rewritten in a slightly different form in order to fit with the framework. The results of \cite{wavbt} are stated at the beginning to display the theorems that compose the starting point. Then, the approximation of the non linear wave equation by finite dimensional problems is detailed. Finally, the first part of the proof of theorem \reff{tttoo} is given, that is, the construction of the measures $\mu$ and $\rho$ and the invariance of $\mu$ under the linear flow.

\smallskip

The second part is mainly analytical, it deals with the local properties of the flow. First, the local existence of the flow is derived from \cite{wavbt}, then a result of local (in time) uniform (for the initial data) convergence of the approched flow towards the local flow of NLW is given, which leads to a result of local invariance of $\rho$ under the local flow.

\smallskip

The last one is dedicated to the extension of the local solution to a global one when the initial data is taken in $\Pi$ and then of the extension of the local invariance result to a global one.

\section{Existence of solution for the cubic NLW}

\subsection{Statement of the main results}

In \cite{wavbt}, Nicolas Burq and Nikolay Tzvetkov have proved that there existed a large subset of $H^\sigma$ with $\sigma < \frac{1}{2}$ that could be taken as initial data for the 3D-non linear wave equation : 

\begin{equation}\label{threednwe} (\partial_t^2  - \Delta )u + u^3 = 0 \end{equation}
with $u$ a radial function and $\Delta$ the Laplace-Beltrami operator on the unit ball of $\R^3$. 

\smallskip

The first paper shows the existence of local solution using a randomization of the initial data.

\smallskip

The randomization is given by :

\begin{definition} Let $s\geq \frac{8}{21}$, $f=(f_1,f_2) \in H^s\times H^{s-1}$ and $\alpha_n$, $\beta_n$ the sequences defined as : 

$$f_1 = \sum_n \alpha_n e_n \; , \; f_2 = \sum_n \beta_n e_n$$
with $e_n$ the eigenfunctions of $\Delta$ on the unit ball with Dirichlet or Neumann conditions. 

\smallskip

Then, let $h_n$, $l_n$ be sequences of real centered gaussian variables, independant from each other on a probability space $\Omega, P$. Set : 

$$f^\omega = (f_1^\omega , f_2^\omega)$$
with

$$f_1^\omega = \sum_n h_n(\omega) \alpha_n e_n \; , \; f_2^\omega = \sum_n l_n(\omega) \beta_n e_n \; .$$
\end{definition}

Then, a local solution exists :

\begin{theorem}Assume $s\geq \frac{8}{21}$ and $f\in H^s \times H^{s-1}$. Set $f^\omega$ defined according to the previous randomization. There exists a regularity parameter $\sigma \geq \frac{1}{2}$ such that for almost all $\omega \in \Omega$, there is a time $T_\omega> 0$ such that there is a unique solution to \reff{threednwe} in

$$\cos (\sqrt{-\Delta }t)f_1^\omega + \frac{\sin(\sqrt{-\Delta}t)}{\sqrt{-\Delta}}f_2^\omega + \mathcal C([-T_\omega,T_\omega], H^\sigma )\; . $$ \end{theorem}

The second one is dedicated to the global extension of these solutions (with Dirichlet boundary conditions). It states :

\begin{theorem}Fix $p\in ]4,6[$. Let $f_1^\omega$ and $f_2^\omega$ be :

$$f_1^\omega = \sum_n \frac{h_n(\omega)}{n\pi} e_n \; , \; f_2^\omega = \sum_n l_n(\omega) e_n \; .$$

Then, for all $s< \frac{1}{2}$ and almost all $\omega \in \Omega$, the problem \reff{threednwe} has a unique global solution in 

$$\mathcal C( \R_t , H^s) \cap L^p_{loc}(\R_t,L^p) \; .$$\end{theorem}

To prove this theorem, a cutoff on the frequencies of the Laplace-Beltrami operator is used. The idea is to solve the equation on finite dimensional functional spaces, spanned by the $N$ first eigenfunctions of the Laplacian. Then, by taking limits of the finite dimensional solutions, a subset of $\Omega$ of full measure appears into which the norms of the local solutions with initial data of the form $f_0(\omega,\; .\; ),f_1(\omega,\; . \; )$ are controlled as the limits of finite dimensional solutions whose norms are themselves controlled. Therefore, this set of full measure provides a set of functions such that the local solution can be extended. The way these sets of control at finite times are built will inspire the construction of other sets onto which not the flow is strongly globally defined but also onto which the measure that we will define is invariant under this flow.

\smallskip

Before going further, the way the problem is reduced to a finite dimensional one will be described, as the definitions involved shall prove themselves useful for the sequel.

\begin{definition} Let $\chi $ be a $\mathcal C_c^\infty$ function with support included in $ [-1,1]$ and satisfying

$$\chi \equiv 1$$
on $[\frac{-1}{2}, \frac{1}{2}]$. Then, for all $N\in \N$ we call $S_N$ the operator $\chi( \frac{-\Delta}{N^2})$ that is to say the operator that maps 

$$\sum_n c_n e_n $$
to

$$\sum_n c_n \chi(\frac{n^2}{N^2}) e_n \; .$$

The set linearly spanned by $\lbrace e_n \; | \; n\leq N \rbrace$ is now called $E_N$ and $\Pi_N$ is the orthogonal projection on $E_N$.
\end{definition}

\begin{proposition}The operators $S_N$ are uniformly continuous from $L^p$ to $E_N$ normed by $L^p$, that is to say that there exists a constant $C$ independant from $N$ such that for all $f\in L^p$,

$$||S_N f||_{L^p} \leq C ||f||_{L^p} \; .$$

Also, for all $f\in L^p$, the sequence $(S_N f)_N$ converges towards $f$ in $L^p$.
\end{proposition}

The proof of this proposition can be found in \cite{strich}.

\smallskip

The reduced problem in finite dimension becomes : 

\begin{equation}\label{findim} \left \lbrace{ \begin{tabular}{ll}
$i\partial_t u + (-\Delta)^{-1/2} u +S_N((S_N \Re u)^3) = 0$ \\
$u|_{t=0} = u_0 = f_1 + i(-\Delta)^{-1/2} f_1 $ \end{tabular} }\right. \; . \end{equation}

This should be explained in the next subsection.

\subsection{Approximation of the flow by finite dimensional problems}

First, one should see how the equation \reff{findim} is derived from the non linear wave equation on the unit ball.

\paragraph{Conserved quantities}

The initial equation is :

\begin{equation}\label{tot} \left \lbrace{ \begin{tabular}{ll}
$\partial_t^2 f -\Delta_{B^3} f +f^3 = 0$ & $t,r \in \R \times B^3$ \\
$f|_{t=0} = f_0 $ & $(\partial_t f)|_{t=0} = f_1$ \end{tabular} }\right. \end{equation}
where $B^3$ is the unit ball of $\R^3$ and $\Delta_{B^3}$ is the Laplace-Beltrami operator on $B^3$ with Dirichlet boundary conditions.

\smallskip

Now, by setting $H = \sqrt{ -\Delta_{B^3}}$, $u_0 = f_0 -i H^{-1} f_1$ and $u= f- iH^{-1}\partial_t f$, $u$ satifies :

\begin{equation}\label{hamilb}  \left \lbrace{ \begin{tabular}{ll}
$i\partial_t u + H u + H^{-1} (\Re u)^3 = 0$ \\
$u|_{t=0} = u_0$ \end{tabular} } \right. \end{equation}

\begin{proposition} The equation \reff{hamilb} is a Hamiltonian equation with energy :

$$\E (u) = \frac{1}{2} \int_{B^3} |H u|^2(r) r^2 dr + \frac{1}{4} \int_{B^3} |\Re u|^4 r^2dr \; .$$
\end{proposition}

The operator $S_N$ is then introduced in order to reduce the problem into an almost finite dimensional one.

\begin{definition} Set $\E_N$ the quantity : 

$$\E_N (u) = \frac{1}{2} \int_{B^3} |H u|^2 r^2dr + \frac{1}{4} |S_N \Re u |^4 r^2 dr \; .$$

This quantity is the hamiltonian of the equation

$$i\partial_t u + H u + S_N( H^{-1} (S_N \Re u)^3) = 0 \; .$$
\end{definition}

\begin{proposition}\label{cauchyprob}Set $u_0 \in H^\sigma$ with $\sigma < \frac{1}{2}$ and $S(t) = e^{iH t}$ the flow of the linear equation $i\partial_t u + Hu =0$ and consider the equation : 

\begin{equation}\label{reduced} \left \lbrace{ \begin{tabular}{ll}
$i\partial_t v + H v + S_N H^{-1}((S_N (S(t)u_0 +v))^3)$ \\
$v|_{t=0} = 0 $ \end{tabular} } \right. \; . \end{equation}

There exists a global strong solution called $v_N$.

Furthermore, $u_N = S(t) u_0 + v_N$ satisfies 

$$i\partial_t u_N + H u_N + S_N( H^{-1} (S_N \Re u_N)^3) = 0 $$
with initial data $u_0$. The flow of this equation is written $u_N(t)=\psi_N (t)(u_0)$.
\end{proposition}

The equation (on $v$) $i\partial_t v + H v + S_N H^{-1}((S_N (S(t)u_0 +v))^3)$ is on $E_N$ and thus the Cauchy-Lipschitz theorem holds and shows that it admits a local unique solution for any initial condition $v_0$, and with $u_0$ fixed. Then, the quantity $\E_N(\Pi_N S(t)u_0 + v)$ does not depend on time and controls $v$, which implies that the local solution does not explode and therefore the solution $v_N$ is global.

\subsection{Building invariant measures}

Now, invariant measures under the flows $\psi_N$ are built. First, call $e_n(r) = \frac{\sin n\pi r}{\sqrt \pi r}$ the eigenfunctions of the Laplacian with Dirichlet boundary conditions. Then, let $\Omega, P$ be a probability space and $(g_n)_n$ a sequence of independant centered and normalized gaussian variables. Set : 

$$\varphi_N (\omega,r) = \sum_{n=1}^N \frac{ g_n(\omega)}{n\pi}e_n(r) \; .$$

The image measure $\mu_N$ of $\varphi_N$ on $E_N$ is absolutely continuous wrt the Lebesgue measure on $E_N$ and : 

$$ d\mu_N(\sum_{n=1}^N (a_n+ib_n e_n)) = d_N \prod_{n=1}^N e^{-(n\pi)^2(a_n^2+b_n^2)/2}da_ndb_n$$

$$ = d_N e^{-\int_{B^3}|H\sum (a_n+ib_n)e_n|^2} \prod_{n=1}^N da_n db_n$$
where $d_N$ is a normalization factor.

\smallskip

Thanks to this point of view, it appears that $\mu_N$ is invariant under the flot $S(t)$ on $E_N$. Indeed, by Liouville theorem, the Lebesgue measure on $E_N$ is invariant under the flow and the quantity $\frac{1}{2} \int |Hu|^2$ is invariant under $S(t)$.

\smallskip

Furthermore, the sequence $\varphi_N$ converges in $L^2_\Omega, H^\sigma_r$ for all $\sigma < \frac{1}{2}$. Denote its limit by $\varphi$ and call $\mu$ the image measure on $H^\sigma$ of $\omega \mapsto \varphi (\omega, .)$.

\bigskip

Also, considering the measure $\mu_N^\bot$ on $E_N^\bot$ (the orthogonal being taken in $H^\sigma$) such that

$$\mu = \mu_N \otimes \mu_N^\bot \; ,$$
it comes that on $E_N^\bot$, $\mu_N^\bot$ is the image measure of 

$$\varphi^N : \omega \mapsto \sum_{n=N+1}^\infty \frac{g_n(\omega)}{n\pi} e_n \; .$$

\begin{lemma} Let $U$ be an open (for the trace topology of $H^\sigma$ on $E_M^\bot$) set of  $E_M^\bot$ and call $\mu_N^M$ the image measure of 

$$\varphi_N^M : \omega \mapsto \sum_{n=M+1}^N \frac{g_n(\omega)}{n\pi}e_n$$
on $E_N^M$ the space linearly spanned by $\lbrace e_{M+1},\hdots , e_N \rbrace$, such that $\mu_M^\bot = \mu_M^N \otimes \mu_N^\bot$. It comes,

$$\mu_M^\bot(U) \leq \liminf_{N \rightarrow \infty} \mu_N^M(U\cap E_N^M) \; .$$

In particular, for $M=0$, this leads to, for all open set $U$ of $H^\sigma$,

$$\mu(U)\leq \liminf_{N\rightarrow \infty} \mu_N (U\cap E_N) \; .$$
\end{lemma}

\begin{proof} Let $\sigma<\sigma_1< \frac{1}{2}$. Let $A$ be the set of $\Omega $, $A= (\varphi^M)^{-1}(U)$ and $A_N = (\varphi^M_N)^{-1}(U\cap E_N^M)$.

\smallskip

If $A$ is empty, then $\mu_M^\bot (U) = 0 = \mu_N^M (U\cap E_N^M)$.

\smallskip

If not, let $\omega \in A$. Since $U$ is an open set, there exists a ball of radius $\epsilon > 0$ such that $\varphi^M(\omega) + B_\epsilon\cap E_M^\bot \subseteq U$. Also,

$$||\varphi_N^M(\omega) - \varphi^M(\omega) ||_{H^\sigma}\leq N^{\sigma-\sigma_1} ||\varphi(\omega)||_{H^{\sigma_1}} .$$

The norm $||\varphi||_{L^2_\omega, H^{\sigma_1}}$ being finite, fot almost all $\omega$, the $||\varphi(\omega)||_{H^{\sigma_1}}$ is finite. So, for almost all $\omega \in A$, there exists $N_0\geq 0$ such that for all $N\geq N_0$, $\varphi_N^M(\omega) \in \varphi^M(\omega)+B_\epsilon \cap E_M^\bot \subseteq U$,  as $\varphi_N^M -\varphi^M (\omega) \in E_M^\bot$, that is there exists $N_0 $ such that for all $N\geq N_0$

$$\omega \in A_N\; \mbox{ that is to say } \omega\in \liminf A_N \; .$$

So, $A \subseteq \liminf A_N$ with the possible exception of a negligible set. By Fatou lemma,

$$\mu_M^\bot(U) = P((\varphi^M)^{-1}(U)) = P(A) \leq P(\liminf A_N) \leq \liminf P(A_N) = \liminf \mu_N^M(U\cap E_N)\; .$$
\end{proof}

\begin{remark}For all closed set $F$ of $E_M^\bot$, 

$$ \mu_M^\bot(F) \geq \limsup \mu_N^M(F\cap E_N^M) \; .$$
\end{remark}

\begin{proposition} The measures $\mu_M^\bot$ are invariant under the flow $S(t)|_{E_M^\bot}$. Therefore, with $M=0$, $\mu$ is invariant under $S(t)$.\end{proposition}

\begin{proof} Let $F$ be a closed set of $E_M^\bot$ and for all $\epsilon > 0$, call $B_\epsilon^M=B_\epsilon \cap E_M^\bot$ with $B_\epsilon$ the open ball of $H^\sigma$ of radius $\epsilon$. For all $t\in \R$, $S(t)$ is a linear isometry of $H^\sigma$ and $E_M^\bot$ is invariant under $S(t)$. Thus, as $F+\overline{B_\epsilon^M}$ is a closed set of $E_M^\bot$, $S(t)F + \overline{B_\epsilon^M} = S(t)(F+\overline{B_\epsilon^M})$ is also closed : 

$$\mu_M^\bot(S(t) F +\overline{B_\epsilon^M}) = \mu(S(t)(F+\overline{B_\epsilon^M}))\geq \limsup \mu_N^M(S(t)(F+\overline{B_\epsilon})\cap E_N^M )$$
and as $S(t)A \cap E_N^M = S(t)(A\cap E_N^M)$,

$$\mu_M^\bot(S(t) F +\overline{B_\epsilon^M})\geq \limsup \mu_N^M (S(t)(F+\overline{B_\epsilon^M}\cap E_N^M))\; .$$

Then, $\mu_N^M$ is invariant under the flow $S(t)|_{E_N^M}$ for the same reasons as $\mu_N$, so

$$\mu_M^\bot(S(t) F +\overline{B_\epsilon^M})\geq \limsup \mu_N^M(F+\overline{B_\epsilon^M}\cap E_N^M) \geq \liminf \mu_N^M(F+\overline{B_\epsilon^M}\cap E_N^M)$$

$$\mu_M^\bot(S(t) F +\overline{B_\epsilon^M})\geq \liminf \mu_N^M(F+B_\epsilon^M \cap E_N^M)\; .$$

As $F+B_\epsilon^M$ is open in $E_M^\bot$,

$$\mu_M^\bot (S(t) F +\overline{B_\epsilon^M})\geq \mu_M^\bot (F+B_\epsilon^M) \geq \mu_M^\bot (F)$$
and by the dominated convergence theorem when $\epsilon \rightarrow 0$,

$$\mu_M^\bot(S(t)F) \geq \mu_M^\bot (F)\; .$$

The linear equation is reversible on all $E_M^\bot$ and $S(t)F$ is closed so, 

$$\mu_M^\bot(F) = \mu_M^\bot (S(-t)S(t)F) \geq \mu_M^\bot(S(t)F)$$
which gives

$$\mu_M^\bot (F) = \mu_M^\bot (S(t)F)$$
for all time $t$ and all closed set $F$.

\smallskip

Then, again because $S(t)$ is an isometry on $E_M^\bot $ and thus preserves the topology, this equality is stable under the passage to the complementary and to denombrable union. Therefore, this property is true for all measurable set $A$ and all time $t$.\end{proof}

\smallskip

As the quantity $\frac{1}{4} \int_{B^3} |S_N \Re u|^4 $ is $\mu$ almost surely finite (see \cite{wavbt}) the measure 

$$d\rho_N (u) = e^{- \frac{1}{4} \int_{B^3}|S_N \Re u|^4}d\mu(u)$$
is well-defined on all $H^\sigma$.

\begin{proposition}\label{fininv} The measure $\rho_N$ is invariant under the flow $\psi_N$ globally defined in proposition \reff{cauchyprob} $\psi_N(t) : H^\sigma \rightarrow H^\sigma$. \end{proposition}

\begin{proof}Consider a measurable set $A$ of initial data $u_0$. For each $u_0$ in $A$, we can write : 

$$u_0 = \Pi_N u_0 + \Pi_N^\bot u_0$$
where $\Pi_N^\bot$ is the orthonormal projector (in $H^\sigma$) on $E_N^\bot$. It suffices to consider $A$ of product type, that is of the type : 

$$A = \lbrace u_0 \; | \; \Pi_N u_0 \in B \; ,\; \Pi_N^\bot u_0 \in C \rbrace$$
with $B$ and $C$ measurable sets of respectively $E_N$ and $E_N^\bot$ since the topology (and so the measurable sets) of $H^\sigma$ is the same as the one of the cartesian product $E_N \times E_N^\bot$. 

\smallskip

Therefore, 

$$\psi_N(t)u_0 = S(t)\Pi_N u_0 + S(t) \Pi_N^\bot u_0 + v(t) = \psi_N|_{E_N}(t) (\Pi_N u_0) + S(t)|_{E_N^\bot} \Pi_N^\bot u_0$$ 

and thus

$$\psi_N(t)(A) = \psi_N|_{E_N}(t) (B) \times S(t)_{E_N^\bot} (C) \; .$$

So, the invariance of $\rho_N$ under $\psi_N$ is reduced to the invariance of $\mu_N^\bot$ under $S(t)$ and the invariance of $e^{- \frac{1}{4} \int_{B^3}|S_N \Re u|^4}d\mu_N(u)$ (on $E_N$) under $ \psi_N|_{E_N}$. The first invariance has already been dealt with. For the second one, all $B\subseteq E_N$ measurable satisfies: 

$$\int_{\psi_N|_{E_N}(t)(B)} e^{-\frac{1}{4} \int_{B^3}|S_N \Re u|^4}d\mu_N (u)=\int_{\psi_N|_{E_N}(t)(B)} e^{-\frac{1}{2}\int_{B^3}|Hu|^2 -\frac{1}{4} \int_{B^3}|S_N \Re u|^4}dL_N (u)$$

$$=\int_{\psi_N|_{E_N}(t)(B)} e^{-\E_N(u)}dL_N(u)$$
where $L_N$ is the Lebesgue measure on $E_N$. By Liouville theorem,  $L_N$ is invariant under $\psi_N|_{E_N}$, therefore the following change of variable $u=\psi_N|_{E_N}(t)(w)$ holds : 

$$\int_{\psi_N|_{E_N}(t)(B)} e^{-\frac{1}{4} \int_{B^3}|S_N \Re u|^4}d\mu_N (u) = \int_{B} e^{\E_N(\psi_N|_{E_N}(t)(w))}dL_N(w)\; .$$

Then, remarking that on $E_N$, $\E_N(\psi_N|_{E_N}(t)(w))$ can be derived over $t$ and is equal to $\E_N(w)$, the measure is invariant and so, $\rho_N$ is invariant under $\psi_N$.

\end{proof}

\begin{definition}Let $f_N$ and $f$ be the application on $H^\sigma$ defined as : 

$$f_N (u)= e^{-\frac{1}{4}\int_{B^3} |S_N \Re u |^4} \mbox{ and } f(u) = e^{-\frac{1}{4}\int_{B^3}|\Re u|^4} \; .$$
\end{definition}

The following statement comes from the analysis of \cite{lestrois}.

\begin{proposition}\label{lunconv} The quantity 

$$\frac{1}{4} \int_{B^3}|\Re u|^4$$
is finite for $\mu$-almost all $u\in H^\sigma$. 

Besides, $f_N$ converges towards $f$ in $L^1_\mu$ norm.
\end{proposition}

Therefore, the measure $\rho$ can be introduced as :

\begin{proposition} The measure $\rho$ such that : 

$$d\rho(u) = f(u)d\mu(u)$$
is well defined and non trivial. And for all $A$ measurable, 

$$\rho(A) = \lim_{N\rightarrow \infty} \rho_N(A)\; .$$

\end{proposition}

The proof of the convergence is very similar to the one in the case of the defocusing NLS, as can be found in \cite{lestrois}, and then :

$$\rho(A) = \int_{A} f(u) d\mu(u) $$

$$|\rho(A) -\rho_N(A) | \leq \int_{A}|f(u)-f_N(u)|d\mu(u) \leq ||f-f_N||_{L^1_\omega} \; .$$

The fact that there exists a set of full $\rho$ measure onto which the flow of \reff{hamilb} is well-defined has been proved in \cite{wavbt}. 

\smallskip

Now, the fact that the measure $\rho$ is invariant under the flow shall be seen.

\section{Uniform convergence of the approached flows}

\subsection{Toolbox}

\paragraph{Sobolev embedding}

For a start, here is the fondamental Sobolev embedding theorem on $\R^n$.

\begin{theorem} Let $n\in \N$ and $s\in \R$. Set $p\in [2,\infty [$ such that $\frac{1}{2} = \frac{1}{p}+\frac{s}{n}$. The functional space $H^s(\R^n)$ is continuously embedded into $L^p(\R^n)$. That is to say, there exists a constant $C(s)$ such that for all $f\in H^s(\R^n)$,

$$||f||_{L^p} \leq C ||f||_{H^s} \; .$$
\end{theorem}

\begin{remark}By considering $f$ radial and with compact support on $B^3$ the unit ball in dimension $3$, as a particular case of the precedent theorem for all $f$ radial and with compact support on $B^3$ and in $H^s(\R^3)$, that is to say for all $f\in \Hs^s$, it comes :

$$||f||_{L^p(B^3)} \leq C ||f||_{H^s(B^3)}$$
as long as $\frac{1}{2} = \frac{1}{p}+\frac{s}{3}$.
\end{remark}

The proof of Sobolev embedding thorem can be found in \cite{aubin}.

\paragraph{Deep into the local existence of solution for the cubic NLW}

The goal here is to show that on certain sets, the flows $\psi_N$ converges uniformly towards $\psi$. So, first, determistic Strichartz estimates and needed properties of the flow are described.

\begin{definition}Let $p > 2$, $q$ such that $\frac{1}{p}+\frac{1}{q} = \frac{1}{2}$, $T>0$ and $s= \frac{2}{p}$. Call

$$X^s_T = \mathcal C^0([-T,T],H^s(B^3))\cap L^p((-T,T),L^q(B^3))$$
where $B^3$ is the unit ball in $\R^3$ and

$$Y^s_T = L^1([-T,T],H^{-s}(B^3))+L^{p'}((-T,T),L^{q'}(B^3))$$
its dual where $p'$ and $q'$ are the conjugate numbers of $p$ and $q$.
\end{definition}

\begin{proposition} Let $p\in ]4,6[$ and $s= \frac{3}{2}-\frac{4}{p}$. There exists a constant $C$ such that for all $T\in ]0,1]$ and all $f$,

$$||f||_{L^p([0,T],L^p(B^3))} \leq C ||f||_{X^s_T} \mbox{ and } ||f||_{Y^s_T} \leq C ||f||_{L^{p'}([-T,T]\times B^3)} \; .$$
\end{proposition}

The proof comes from a particular case of interpolation between the two functional spaces described in the definition of $X^s_T$.

Thanks to a combination of Sobolev embedding theorem and Strichartz inequality (see \cite{kapit} for further details), the following property holds : 

\begin{proposition}\label{combi} Let $p \in ]4,6[$ and $s$ defined as $s=\frac{3}{2} - \frac{4}{p}$, there exists $C\leq 0$ such that for all $T \in [0,1]$ and all $f \in H^s$ : 

$$||S(t)f||_{L^p([-T,T]\times B^3)}\leq C ||f||_{H^s} \; .$$
\end{proposition}

\begin{proposition}\label{ineq} Let $p\in ]4,6[$, $p_1\in ]4,6[$ such that $p_1 > p$, $s=\frac{3}{2}-\frac{4}{p}$ and $s_1 = \frac{3}{2}-\frac{4}{p_1} > s$. There exists $C$ such that for all $T\in ]0,1]$ and all $f$,

$$\int_{0}^t S(t-u)H^{-1}f(u) du||_{X^s_T}\leq C ||f||_{Y_T^{1-s}}$$

$$||(1-S_N)\int_{0}^t H^{-1} S(t-u) f(u) du ||_{X^s_T} \leq C N^{s-s_1} ||f||_{Y^{1-s_1}_T} \; .$$
\end{proposition}

The proof can be found in \cite{wavbt}. The last crucial result needed from this article is the local existence theorem, and its implication regarding the $X^s_T$ norms of the function $v(t)=\psi(t)u_0 - S(t) u_0$ where $\psi(t)$ would be defined as the flow of 

$$i\partial_t u + Hu + H^{-1} (\Re u)^3 = 0$$
that is to say $v$ is the solution of

\begin{equation}\label{limofred} \left \lbrace{\begin{tabular}{ll}
$i\partial_t v + H v + H^{-1}((\Re(S(t)u_0 + v(t)))^3) = 0$ \\
$v|_{t=0} = 0$ \end{tabular} }\right. \; . \end{equation}

\begin{theorem}\label{locex} Choose a real number $p \in ]4,6[$ and define $s$ as $s=\frac{3}{2}-\frac{4}{p}$. There exists $C> 0$, $c>0$, $\gamma = 1-\frac{4}{p}$ such that for any arbitrary large number $A$, there exists a time of existence $\tau \in ]0,1]$ depending on $A$ as $\tau = c(1+A)^{-\gamma}$ such that for all initial data $u_0$ satisfying $||S(t)u_0||_{L^p_{t,x}\in [0,2]\times B^3} \leq A$, there exist unique solutions of the equations \reff{reduced} and \reff{limofred}, $v_N$ and $v$, belonging to $X^s_\tau$ and satisfying : 

$$||v||_{X^s_\tau },||v_N||_{X^s_\tau} \leq C A \; .$$

Also, as $S(t)$ is $2$ periodic (the eigenvalues of the Laplacian on $B^3$ with Dirichlet boundary conditions are of the form $(n\pi)^2$, $n\in \N^*$) and thanks to the proposition \reff{combi}, there exists another constant $C'$ such that for each $t \in [-\tau, \tau]$ :

$$||S(t')(u(t))||_{L^p_{t',x}},||S(t')(u_N(t))||_{L^p_{t',x}} \leq ||S(2t')u_0||_{L^p_{t',x}}+||S(t')v(t)||_{L^p_{t',x}} \leq C'A $$
and if $u_0 \in H^\sigma$, the solutions satisfies : 

$$||u(t)||_{H^\sigma},||u_N(t)||_{H^\sigma} \leq C' ||u_0||_{H^\sigma}$$
with $u(t) = S(t)u_0 +v$ and $u_N(t) = S(t)u_0 + v_N$.

\end{theorem}

\paragraph{Remarks on sets' measurements} 

The local results of existence will provide local properties of uniform convergence of the sequence of flows $\psi_N$ toward the flow $\psi$ and then induce properties on the invariance of the flow that will remain local. In order to extend those next to appear local properties into a global invariance of the flow, we will have to control the quantity denoted as $A$ in the previous theorem. But this control has to satisfy certain properties, such as the set that describes the initial data that lead to a controlled solution must be of full measure.

\smallskip

To this purpose, consider the following proposition.

\begin{proposition}\label{majd} Let $\sigma < \frac{1}{2}$, let $p\in ]4,6[$, let $D\geq 0$ and consider the sets : 

$$B(D)^c = \lbrace u_0 \in H^\sigma \; |\; ||S(t)u_0||_{L^p_{t,x}} > D \; \rbrace $$
and

$$E(D)^c = \lbrace u_0 \in H^\sigma \; |\; ||u_0||_{H^\sigma }> D \rbrace \; .$$

There exists $c>0$ independant from $D$ such that : 

$$\rho(B(D)^c) , \rho_N(B(D)^c) \leq \mu(B(D)^c) \leq e^{-c D^2}$$
and

$$\rho(E(D)^c), \rho_N(E(D)^c) \leq \mu(E(D)^c) \leq e^{-cD^2} \; .$$

\end{proposition}

The proof depends on the lemma 3.3 that can be found in \cite{lestrois}.

\begin{remark}It will appear that the time $\tau_1 < \tau$ such that there is local convergence on $X^s_{\tau_1}$ depends on $D$. It will then be necessary to prove that $\tau_1$ is big enough to control $u(t)$ at some finite times $t_k$ with $k\in \varmathbb Z$ cover all times, and still have a set of initial data of full $\rho$ measure. \end{remark}

\subsection{Local uniform convergence}

We now want to prove that the flows $\psi$ and $\psi_N$ are such that $\psi(t)u_0-\psi_N(t)u_0$ converges in $X_{\tau_1}^s$ for some $\tau_1$ uniformly in $u_0$.

\begin{lemma}\label{loccg} Let $\sigma\in ]0,\frac{1}{2}[$ ,$p \in ]4,6[$ and $s$ defined as $s=\frac{3}{2}-\frac{4}{p}$. Fix $D\geq 0$ and consider $A(D)$ the set

$$A(D) = \lbrace u_0 \in H^\sigma \; | \; ||S(t)u_0||_{L^p}\leq D \mbox{ and } ||u_0||_{H^\sigma}\leq D \rbrace .$$

There exists $c_1 > 0$ and $\gamma_1 > 0$ such that by fixing $\tau_1 = \min (c_1(1+D)^{-\gamma_1},\tau)$, where $\tau$ is the time provided by the theorem \reff{locex} for all $\epsilon > 0$, there exists $N_0 \geq 0$ such that for all $u_0 \in A(D)$ and all $N \geq N_0$,

$$||\psi(t)u_0 - \psi_N(t)u_0||_{X^s_{\tau_1}}< \epsilon \; ,$$
that is to say that $\psi_N(t)u_0$ converges uniformly in $u_0 \in A(D)$ in $X^s_{\tau_1}$.
\end{lemma}

\begin{proof} Let $u_0 \in A(D)$ and $v$ and $v_N$ be such that 

$$\psi(t)u_0 = S(t) u_0 + v(t)  \mbox{ and } \psi_N(t)u_0 = S(t)u_0 + v_N(t)\; .$$

The functions $v$ and $v_N$ are both in $X^s_{\tau}$ so the norm

$$||\psi(t)u_0 -\psi_N(t)u_0||_{X^s_{\tau_1}} = ||v-v_N||_{X^s_{\tau_1}}$$
is finite for all $\tau_1 \leq \tau $.

\smallskip

For all $t\leq \tau$ : 

$$v(t)-v_N(t) = \int_{0}^t S(t-s) H^{-1}\left( (\Re \psi(s)u_0)^3-S_N ((S_N \Re \psi_N(s)u_0 )^3) \right) ds $$
that is to say $v-v_N = I_N + II_N$ with

$$I_N = (1-S_N)\int_{0}^t S(t-s) H^{-1}\left( (\Re \psi(s)u_0)^3 \right)ds $$
and

$$II_N = \int_{0}^t S(t-s) S_N H^{-1}\left( (\Re \psi(s)u_0)^3- (\Re S_N \psi_N(s)u_0)^3 \right) ds \; .$$

Thanks to proposition \reff{ineq}, for $T \leq \tau$ and $s_1 > s$

$$||I_N||_{X^s_{T}} \leq C N^{s-s_1} ||(\Re \psi u_0)^3||_{Y^{1-s_1}_T}$$
with $C$ independant from $D$, $T$, and $N$. Let $p_1$ be such that $1-s_1 = \frac{3}{2}-\frac{4}{p_1}$. The condition $s_1>s$ is equivalent to $p_1<\frac{2p}{p-2}$. Hence :

$$||I_N||_{X^s_T} \leq C N^{s-s_1} ||(\Re \psi u_0 )^3||_{L^{p_1'}_{t,x}} $$

$$||I_N||_{X^s_T} \leq C N^{s-s_1} ||(\Re \psi u_0 )||_{L^{3p_1'}}^3 \; .$$

To majore this norm, the condition $3p_1' \leq p $ is wanted. This condition is equivalent to $p_1 \geq \frac{p}{p-3}$, which means  $p_1$ has to be chosen in the interval $[\frac{p}{p-3}, \frac{2p}{p-2}[$. But since $p> 4$, $\frac{p}{p-3}<\frac{2p}{p-2}$, and so, such a choice is possible, and in particular, by choosing $p_1 = \frac{p}{p-3}$, or $3p'_1 = p$, it comes : 

$$||I_N||_{X^s_T} \leq C N^{s-s_1} \left( ||S(t)u_0 ||_{L^p} + ||v||_{L^p}\right)^3$$

$$||I_N||_{X^s_T} \leq C N^{s-s_1} \left( ||S(t)u_0 ||_{L^p} + ||v||_{X^s_T} \right)^3$$
and thanks to theorem \reff{locex}, $||v||_{X^s_T}\leq CD$ with $C$ independant from $D$ and $T$ as long as $T\leq \tau $ so

$$||I_N||_{X^s_T} \leq C N^{s-s_1} D^3 \; .$$

Therefore, for all $\epsilon > 0$, there exists $N_0$ such that for all $u_0 \in A(D)$, all $T\leq \tau$ and all $N\geq N_0$,

$$||I_N||_{X^s_T} \leq \epsilon \; .$$

Once more, thanks to \reff{ineq}, 

$$||II_N||_{X^s_T} \leq ||\left( (\Re \psi(s)u_0)^3- (\Re S_N \psi_N(s)u_0)^3 \right)||_{Y^{1-s}_T}$$
and with $p_2$ such that $1-s = \frac{3}{2}-\frac{4}{p_2}$ that is $p_2 = \frac{2p}{p-2}$, 

$$||II_N||_{X^s_T} \leq ||\left( (\Re \psi(s)u_0)^3- (\Re S_N \psi_N(s)u_0)^3 \right)||_{L^{p'_2}} \; .$$

Since 

$$|(\Re \psi(s)u_0)^3- (\Re S_N \psi_N(s)u_0)^3| \leq \frac{3}{2}|\Re \psi(s)u_0-\Re S_N\psi_N(s)u_0|\left( (\Re \psi(s)u_0)^2+(\Re S_N\psi_N(s)u_0)^2 \right) \; ,$$

by Hölder inequality with $\frac{1}{p_2'} = \frac{1}{3p'_2}+\frac{2}{3p'_2}$,

$$||II_N||_{X^s_T} \leq ||\Re \psi(s)u_0- \Re S_N \psi_N(s)u_0||_{L^{3p'_2}}||(\Re \psi(s)u_0)^2+(\Re S_N\psi_N(s)u_0)^2 ||_{L^{3p'_2/2}}$$

$$\leq ||\Re \psi(s)u_0- \Re S_N \psi_N(s)u_0 ||_{L^{3p'_2}}\left( ||(\Re \psi(s)u_0)^2||_{L^{3p'_2/2}}+||(\Re S_N \psi_N(s)u_0)^2||_{L^{3p'_2/2}}\right) $$

As $3p'_2 = \frac{6p}{p+2} < p$ and 

$$||(\Re \psi(s)u_0)^2||_{L^{3p'_2/2}} = ||\Re \psi(s)u_0||^2_{L^{3p'2}} \leq T^{2\gamma_2} ||\psi(s)u_0||_{L^p}^2$$
with $\gamma_2 = \frac{p-4}{6p}$, it comes that : 

$$||II_N||_{X^s_T}\leq C D^2 T^{2\gamma_2}||\Re \psi(s)u_0- \Re S_N \psi_N(s)u_0||_{L^{3p'_2}}$$
with $C$ independent from $D$ and $T$ as long as $T\leq \tau$.

\smallskip 

Now, the quantity $||\Re \psi(s)u_0- \Re S_N \psi_N(s)u_0||_{L^{3p'_2}}\leq \alpha_N + \beta_N$ remains to be considered, with $\alpha_N = ||(1-S_N) \psi(s) u_0||$ and $\beta_N = ||S_N(\psi(s)u_0 -\psi_N(s)u_0)||$.

By the same convex inequalities as precedently,  

$$\beta_N \leq C T^{\gamma_2} ||\psi(s)u_0-\psi_N(s)u_0||_{L^p}\leq CT^{\gamma_2} ||\psi(s)u_0-\psi_N(s)u_0||_{X^s_T}\; .$$

Choose $2<p_3 < \frac {3}{3/2-\sigma} \in ]2,3[$  and call $\sigma_3 = 3\left( \frac{1}{2}-\frac{1}{p_3} \right)< \sigma $. As $p_3< 3 < 3p'_2 < p$ so, there exists $\theta \in ]0,1[$ such that $\frac{1}{3p'_2} = \frac{\theta}{p_3}+\frac{1-\theta}{p}$, thus

$$\alpha_N \leq ||(1-S_N)\psi(s)u_0 ||_{L^{p_3}}^\theta ||(1-S_N) \psi(s)u_0||_{L^p}^{1-\theta}$$

$$||(1-S_N) \psi(s)u_0||_{L^p}\leq C(||S(t)u_0||_{L^p}+||v||_{X^s_T})\leq CD$$
and by Sobolev embedding theorem : 

$$||(1-S_N)\psi(s)u_0||_{L^{p_3}}\leq ||(1-S_N)\psi(s)u_0||_{L^{p_3}_t,H^{\sigma_3}_x}$$

Now, for all $s$,

$$||(1-S_N)\psi(s) u_0||_{H^{\sigma_3}} \leq C N^{\sigma_3-\sigma }||\psi(s)u_0||_{H^\sigma} \leq C D N^{\sigma_3-\sigma}$$
and so

$$\alpha_N \leq C N^{\theta(\sigma_3-\sigma)}D^\theta D^{1-\theta} = CN^{\theta (\sigma_3-\sigma)}D \; .$$

Therefore, for all $\epsilon > 0$, there exists $N_0$ such that for all $u_0\in A(D)$, all $T\leq \tau$, and all $N\geq N_0$, 

$$\alpha_N \leq \epsilon \; .$$

Now, let us sum up these inequalities.

$$||\psi(t)u_0-\psi_N(t) u_0||_{X^s_T}\leq I_N + C D^2 T^{2\gamma_2}(\alpha_N + \beta_N) \leq I_N + C\alpha_N + CD^2 T^{3\gamma_2} ||\psi(t)u_0-\psi_N(t) u_0||_{X^s_T}\; .$$

Set $ \gamma_1 = \max (\gamma, \frac{2}{3\gamma_2})$ and $\tau_1 = \min(\tau, C^{-1/(3\gamma_2)}(1+D)^{-\gamma_1})$, such that $CD^2 \tau_1^{3\gamma_2} < CD^2 \tau_1^{2/\gamma_1}<1$, hence

$$||\psi(t)u_0-\psi_N(t) u_0||_{X^s_{\tau_1}}\leq C (I_N + \alpha_N)$$
so for all $\epsilon > 0$ there exists $N_0$ such that for all $u_0 \in A(D)$ and all $N\geq N_0$

$$||\psi(t)u_0-\psi_N(t) u_0||_{X^s_{\tau_1}}\leq \epsilon \; .$$
\end{proof}

\begin{remark}Note that the construction of $\gamma_1 $ ensures that $\gamma_1 \geq \gamma$ but that $ \tau_1$ is still a power of $D$. \end{remark}

\subsection{Local invariance}

Let us show that the measure is invariant under the flow locally in time. That is, as long as the sequence $\psi_N(t)u_0$ converges uniformly on some sets, it will appear that $\rho(\psi(t)A) \geq \rho(A)$. This is the first step in order to reach a global invariance result for the measure. 

\begin{lemma}\label{locgrowth} Let $\sigma \in ]0,\frac{1}{2}[$, $p\in ]4,6[$, $s = \frac{3}{2}- \frac{4}{p}$, and $D> 0$. Set $A(D)$ the set described in \reff{loccg}, $\tau = c(1+D)^{-\gamma}$ the local exitence time coming from theorem \reff{locex} and $\tau_1 = \min(\tau, c_1(1+D)^{-\gamma_1})$ the local time of uniform convergence, all three depending only on $D$ and $p$. Then, for all $A\subseteq A(D)$ measurable, and all $t\in [-\tau_1,\tau_1]$, the set $\psi(t) A$ is measurable and : 

$$\rho(\psi(t)A) = \rho(A) \; .$$
\end{lemma}

\begin{proof}First, for all $A$ measurable, $\psi(t) A$ is also measurable thanks to the local continuity of the flow. Assume now that $A$ is a closed set of $H^\sigma$ included in $A(D)$ and set $\epsilon > 0$. By lemma \reff{loccg}, there exists $N_0$ such that for all $u_0 \in A(D)$ and all $N\geq N_0$

$$||\psi(t)u_0 - \psi_N(t)u_0||_{X^s_{\tau_1}}\leq \epsilon \; .$$

But by definition, $||\; . \; ||_{X^s_{\tau_1}}\geq ||\; .\; ||_{\mathcal C^0([-\tau_1,\tau_1],H^s(B^3))}$. So for all $t\in [-\tau_1,\tau_1]$, 

$$||\psi(t)u_0-\psi_N(t)u_0||_{H^s} \leq \epsilon \; .$$

Let $B_{\epsilon}$ be the ball in $H^s$ of center $0$ and radius $\epsilon $, as $A\subseteq A(D)$, for all $N\geq N_0$, and all $t\in [-\tau_1,\tau_1]$,

$$\psi_N(t)(A) \subseteq \psi(t)(A) + B_{\epsilon}$$
therefore

$$\rho_N(\psi_N(t)(A)) \leq \rho_N(\psi(t)(A)+B_\epsilon) \; .$$

Then, since the measure $\rho_N$ is invariant under the flow $\psi_N$, as it is stated in proposition \reff{fininv},

$$\rho_N(\psi_N(t)A) = \rho_N(A)$$
and thus

$$\rho_N(A) \leq \rho_N(\psi(t)A + B_\epsilon) \; .$$

Then, using the fact that (proposition \reff{lunconv}), for all $\mu$ measurable set $B$, 

$$\rho(B) = \lim_{N\rightarrow \infty} \rho_N(B)\; ,$$
and as the property $\rho_N(A) \leq \rho_N(\psi(t)A + B_\epsilon)$ is true for all $N\geq N_0$, by taking the limit : 

$$\rho(A) \leq \rho(\psi(t) A + B_{\epsilon})$$
and then by making $\epsilon$ tend toward $0$, thanks to the dominated convergence theorem,

$$\rho(A) \leq \rho(\psi(t) A)$$
and that for all $t\in [-\tau_1,\tau_1]$. Indeed, $\psi(-t) = \psi(t)^{-1}$ is continuous in $H^\sigma$, so $\psi(t)A$ is closed in $H^\sigma$.

\smallskip

For the reverse inequality, use the fact that indeed, $\psi(t)A \subseteq \psi_N(t) A + B_\epsilon$ for all $u_0 \in A(D)$ and $ n\geq N_0$. It is also true that the ball $\widetilde B_\epsilon$ of radius $\epsilon$ in $H^\sigma$ contains $B_\epsilon$ as $\sigma<\frac{1}{2}< s$, so :

$$\psi(t) A \subseteq \psi_N(t)A + \widetilde B_\epsilon \; .$$

Then, the fact that the equation is reversible, and so $\psi_N(t)^{-1} = \psi_N(-t)$ is used. Also, thanks to the continuity of the local flow on $H^\sigma$, there exists a constant $C$ depending on the time $t$ but not on $\epsilon$ or $N$ such that : 

$$\psi_N(-t)(\psi_N(t) A + \widetilde B_\epsilon) \subseteq A + \widetilde B_{C\epsilon}$$
so

$$\psi(t) A \subseteq \psi_N(t)A+\widetilde B_\epsilon \subseteq \psi_N(t)(A + \widetilde B_{C\epsilon})$$
and

$$\rho_N(\psi(t)A) \leq \rho_N \left( \psi_N(t)(A + \widetilde B_{C\epsilon})\right) = \rho_N(A+\widetilde B_{C\epsilon})$$
thanks to the invariance of $\rho_N$ under $\psi_N$.

\smallskip

By passing to le limit $N\geq N_0 \rightarrow \infty$,

$$\rho(\psi(t) A) \leq \rho(A+\widetilde B_{C\epsilon})$$
and then $\epsilon \rightarrow 0$,

$$\rho(\psi(t)A) \leq \rho(A) \; $$
so for all closed set $A$ of $H^\sigma$ included in $A(D)$, 

$$\rho(\psi(t)A) = \rho(A)\; .$$

Then, remark that $A(D)$ is a closed set of $H^\sigma$, so

$$\rho(\psi(t)A(D)) = \rho(A(D)) $$
and thus the property of invariance under the flow passes to the complementary and the denombrable unions. It holds on every measurable set.
\end{proof}

\section{Measure invariance}

\subsection{Building sets of full measure with global existence}

We now need to build a set of full $\rho$ measure such that in this set not only the local invariance result holds but also can be extended to a global one, that is to say that in this set the flow must be globally well defined.

\begin{definition}Let 

$$D_{i,j} = ( i + j^{1/\gamma_1})^{1/2}$$ 
with $i,j \in \N$ and set $T_{i,j} = \sum_{l=1}^j \tau_1(D_{i,l})$. Let

$$\Pi_{N,i} = \lbrace u_0 \; |\; \forall j\in \N \; , \; \psi_N(\pm T_{i,j} ) \in A(D_{i,j+1})\rbrace $$
and

$$\Pi_i = \limsup_{N\rightarrow \infty} \Pi_{N,i}$$
and finally 

$$\Pi = \bigcup_{i\in \N} \Pi_i \; .$$
\end{definition}

\begin{proposition} The set $\Pi$ is of full $\rho$ measure. \end{proposition}

\begin{proof} Let us compute the measure of the complementary set of $\Pi$. 

\smallskip

First, as 

$$\Pi_{N,i} = \bigcap_{j\in \N} \psi_{N}(\pm T_{i,j})^{-1} (A(D_{i,j+1}))\; ,$$

$$\rho(\Pi_{N,i}^c) \leq ||f-f_N||_{L^1_\mu} + \rho_N(\Pi_{N,i}^c) $$
and

$$\rho_N(\Pi_{N,i}^c)\leq \sum_{j=0}^\infty \rho_N(\left( \psi_N(\pm T_{i,j})^{-1}(A(D_{i,j+1})\right)^c) \; .$$

Then, using that $\left( \psi_N(\pm T_{i,j})^{-1}(A(D_{i,j+1})\right)^c = \psi_N(\pm T_{i,j})^{-1}(A(D_{i,j})^c)$, it becomes clear that the $\rho_N$ measures of these sets are equal and as $\rho_N$ is invariant under the flow $\psi_N$ : 

$$\rho_N(\Pi_{N,i}^c) \leq 2 \sum_{j} \rho_N(A(D_{i,j})^c) \leq 2 \sum_{j} \mu(A(D_{i,j})^c) $$

But, $A(D)^c = B(D)^c \cup E(D)^c$ with $B$ and $E$ the sets defined in \reff{majd} so

$$\mu(A(D)^c) \leq \mu(B(D)^c) + \mu(E(D)^c) \leq 2 e^{-cD^2}$$
so

$$ \rho_N(\Pi_{N,i}) \leq C \sum_j e^{-cD_{i,j}^2}$$
and so

$$\rho(\Pi_{N,i}^c) \leq Ce^{-c i}\sum_{j} e^{-c j^{1/\gamma_1}}\leq C e^{-c i}$$
as $e^{-cj^{1/\gamma_1}}$ is the general term of a convergent series.

\smallskip

Therefore : 

$$\rho(\Pi_i^c) = \rho(\liminf \Pi_{N,i}^c) \leq \liminf \rho(\Pi_{N,i}^c) =  e^{-cD^2i} + \liminf ||f-f_N||_{L^1_\mu} = C e^{-ci}$$
and then

$$\rho(\Pi^c) = \rho(\bigcap_{i} \Pi_i^c) \leq \lim \rho(\Pi_i^c) = 0$$
that is to say that $\Pi$ is of full measure.
\end{proof}

\begin{lemma}\label{discon}Let $u_0 \in \Pi_i$, the flow $\psi(t)u_0$ is strongly globally defined and for all $j \in \N$ and $\psi(\pm T_{i,j}) u_0$ belongs to $A(D_{i,j+1})$. \end{lemma}

\begin{proof} As $u_0 \in \Pi_i = \limsup \Pi_{N,i}$, there exists a sequence $N_k \rightarrow \infty$ such that for all $k$, $u_0 \in \Pi_{N_k,i}$, which is equivalent to $\psi_{N_k}(\pm T_{i,j})u_0 \in A(D_{i,j})$ for all $j \in \N$.

\smallskip

Then, by recurrence over $j$, it can be proved that $\psi(t) u_0$ is defined on $[-T_{i,j},T_{i,j}]$ and that $ \psi(\pm T_{i,j})u_0 -\psi_{N_k}(\pm T_{i,j})u_0$ converges toward $0$ in $\Hs^s$ when $k\rightarrow \infty$.

\smallskip

For $j=0$, $T_{i,0}= 0$ and so $u_0 = \psi_{N_k}(\pm T_{i,0})u_0 \in A(D_{i,1}$ and $\psi(T_{i,0})u_0- \psi_{N_k}(T_{i,0})u_0 = 0$ converges toward $0$ in $H^s$. 

\smallskip

Suppose that at rank $j$, $\psi(t) u_0$ strongly exists on $[-T_{i,j},T_{i,j}]$ and $\psi(\pm T_{i,j})(u_0)-\psi_{N_k}(\pm T_{i,j})(u_0)$ converges toward $0$  in $H^s$. Let us show that the property holds at rank $j+1$. As 

$$||\psi(\pm T_{i,j})(u_0)-\psi_{N_k}(\pm T_{i,j})(u_0)||_{H^\sigma} \leq ||\psi(\pm T_{i,j})(u_0)-\psi_{N_k}(\pm T_{i,j})(u_0)||_{H^s}$$
and

$$||S(t)\left(\psi(\pm T_{i,j})(u_0)-\psi_{N_k}(\pm T_{i,j})(u_0)\right) ||_{L^p_{t,x}} \leq ||\psi(\pm T_{i,j})(u_0)-\psi_{N_k}(\pm T_{i,j})(u_0)||_{H^s}$$

and for all $k$,

$$||\psi_{N_k}(\pm T_{i,j})(u_0)||_{H^\sigma} \leq D_{i,j+1}$$
and

$$||S(t)\left(\psi_{N_k}(\pm T_{i,j})(u_0) \right) ||_{L^p_{t,x}}\leq D_{i,j+1}\; ,$$

by taking the limits when $k\rightarrow \infty$, it comes that $u_\pm : = \psi(\pm T_{i,j})u_0 \in A(D_{i,j+1})$.

\smallskip

So, thanks to theorem \reff{locex} $\psi(t)u_\pm$ is strongly defined on $[-\tau_1,\tau_1] \subseteq [-\tau,\tau]$ and thanks to lemma \reff{loccg},

$$\psi(t) u_\pm -\psi_{N_k}(t) u_\pm$$

converges toward $0$ in $X^s_{\tau_1(D_{i,j+1})}$. In particular,

$$\psi(\pm \tau_1(D_{i,j+1}) u_\pm - \psi_{N_k}(\pm \tau_1(D_{i,j+1})) u_\pm$$
converges toward $0$ in $H^s$.

\smallskip

Then, as

$$\psi(\pm T_{i,j+1})u_0 - \psi_{N_k}(\pm T_{i,j+1}) u_0 = \psi(\pm \tau_1(D_{i,j+1)})u_\pm -\psi_{N_k}(\pm \tau_1(D_{i,j+1})) u_\pm + $$

$$\psi_{N_k}(\pm \tau_1(D_{i,j+1}))u_\pm - \psi_{N_k}(\pm \tau_1(D_{i,j+1})) (\psi_{N_k}(\pm T_{i,j})u_0 $$
and since $v \mapsto \psi_{N}(\pm \tau_1(D_{i,j+1})) v$ is uniformly in $N$ continuous from $ H^s \cap A(D_{i,j+1})$ to $H^s$, it implies that 

$$\psi_{N_k}(\pm \tau_1(D_{i,j+1}))u_\pm - \psi_{N_k}(\pm \tau_1(D_{i,j+1})) (\psi_{N_k}(\pm T_{i,j})u_0$$
also converges toward $0$ in $H^s$. Therefore,

$$\psi(\pm T_{i,j+1})u_0 - \psi_{N_k}(\pm T_{i,j+1}) u_0$$
converges toward $0$ in $H^s$ and as it has precedently been seen, it implies that $\psi(\pm T_{i,j+1}) \in A(D_{i,j+2})$.
\end{proof}

\subsection{Global invariance}

Now, a first result of global invariance can be proved.

\begin{proposition}Let $A$ be a measurable set included in $\Pi_i$. Then for all $t\in \R$, we have

$$\rho(\psi(t)(A)) = \rho(A) \; .$$
\end{proposition}

\begin{proof}

In order to prove such a fact, it is required that the sequence $T_{i,j}$ where $i$ is fixed diverges.

\smallskip

Indeed,

$$\tau_1(D_{i,j}) = \min(\tau(D_{i,j},c_1(1+D_{i,j})^{-\gamma_1}) = \min (c(1+D_{i,j})^{-\gamma},c_1(1+D_{i,j})^{-\gamma_1})$$
and $D_{i,j}= \sqrt{i+j^{1/\gamma_1}}$ diverges. Therefore, as $\gamma_1 \geq \gamma$, above a certain rank $\tau_1(D_{i,j}) = c_2(1+D_{i,j})^{-\gamma_1}$ with $c_2 = c_1$ if $\gamma < \gamma_1$ or $c_2 = \min(c,c_1)$ otherwise.

\smallskip

So, $\tau_1(D_{i,j})$ behaves like $j^{-1/2}$ when $j \rightarrow \infty$ and so the sequence $T_{i,j}$ diverges.

\smallskip

Let $t\in \R$, there exists $j$ such that $t \in [T_{i,j},T_{i,j+1}]$ if $t\geq 0$ or $t\in [-T_{i,j+1},-T_{i,j} ]$. Let us show by recurrence over $j$ that for all $t \in [T_{i,j},T_{i,j+1}]\cup [-T_{i,j+1},-T_{i,j} ]$,

$$\rho(\psi(t)A) = \rho(A)$$.

For $j=0$, we have $T_{i,0} = 0$, $T_{i,1} = \tau_1(D_{i,1})$ and $A = \psi(T_{i,0}) (A) \subseteq \psi(T_{i,0})(\Pi_i) \subseteq A(D_{i,1})$ thanks to lemma \reff{discon}. So, the local invariance lemma \reff{locgrowth} holds : for all $t\in [-\tau_1(D_{i,1}),\tau_1(D_{i,1})]$,

$$\rho(\psi(t)(A)) = \rho(A)\; .$$

For $j-1\Rightarrow j$,  $\psi(\pm T_{i,j})(A) \subseteq \psi(T_{i,j}) (\Pi_i) \subseteq A(D_{i,j+1})$ (lemma \reff{discon}). So, by using lemma \reff{locgrowth}, for all $t\in [0, \tau_1(D_{i,j+1})]$,

$$\rho\left(\psi(\pm t)\left( \psi(\pm T_{i,j})(A) \right) \right) = \rho\left( \psi(\pm T_{i,j})(A)\right) \; .$$

Then, by using the recurrence hyothesis,

$$\rho\left( \psi(\pm T_{i,j})(A)\right) = \rho(A) \; .$$

And so, for all 

$$t\in [T_{i,j},T_{i,j}+\tau_1(D_{i,j+1})]\cup [-T_{i,j}-\tau_1(D_{i,j+1}),-T_{i,j} ]$$

$$= [T_{i,j},T_{i,j+1}]\cup [-T_{i,j+1},-T_{i,j} ]\; , $$ 

it comes

$$\rho(\psi(t)(A)) = \rho(\psi(t)(A)) \; .$$

\end{proof}

\begin{theorem}For all $\rho$ measurable set included in $\Pi$, we have :

$$\rho(\psi(t)(A)) = \rho(A) \; .$$
\end{theorem}

\begin{proof} As $\Pi = \bigcup_{i}\Pi_i$, and $A\subseteq \Pi$, $A$ can be written : 

$$A= \bigsqcup_{i\in \N} A_i$$
with $A_i\subseteq \Pi_i$, and the $A_i$ disjoint. So,

$$\psi(t)(A) = \bigsqcup_{i\in \N } \psi(t)A_i$$
since the flow is strongly defined in $\Pi$.

$$\rho(\psi(t)A) = \sum_{i\in \N } \rho(\psi(t)A_i) = \sum_{i\in \N} \rho(A_i) = \rho(\bigsqcup_{i\in \N} A_i) = \rho(A)\; . $$

\end{proof}

\bibliographystyle{amsplain}
\bibliography{bibliagain} 
\nocite{*}

\end{document}